\documentclass[11pt]{article}
\usepackage[a4paper]{anysize}\marginsize{1.5cm}{1.5cm}{1.5cm}{1cm}
\pdfpagewidth=\paperwidth \pdfpageheight=\paperheight
\usepackage{amsfonts,amssymb,amsthm,amsmath,eucal,tabu}
\usepackage{pgf}
\usepackage{array}
\usepackage{pstricks}
\theoremstyle{plain}
\newtheorem{thm}{Theorem}[section]
\newtheorem{theorem}[thm]{Theorem}

\newtheorem{lemma}[thm]{Lemma}
\newtheorem{proposition}[thm]{Proposition}

\theoremstyle{definition}
\newtheorem{definition}[thm]{Definition}
\newtheorem{remark}[thm]{Remark}
\newtheorem{example}[thm]{Example}

\newtheorem{thevarthm}[thm]{\varthmname}

\newenvironment{varthm*}[1]{\trivlist\item[]{\bf #1.}\it}{\endtrivlist}

\renewcommand\geq{\geqslant}

\renewcommand\leq{\leqslant}

\newcommand\be{\begin{eqnarray*}}
\newcommand\ee{\end{eqnarray*}}

\newcommand\Q{\mathbb Q}
\newcommand\R{\mathbb R}

\renewcommand\P{\mathbb P}

\newcommand\calo{{\mathcal O}}

\newcommand\newop[2]{\def#1{\mathop{\rm #2}\nolimits}}
\newop\log{log}
\newop\ord{ord}
\newop\Gal{Gal}
\newop\SL{SL}
\newop\Bl{Bl}
\newop\mult{mult}
\newop\mass{mass}
\newop\div{div}
\newop\codim{codim}
\newop\sing{sing}
\newop\vdim{vdim}
\newop\edim{edim}
\newop\Ass{Ass}
\newop\size{size}
\newop\reg{reg}
\newop\satdeg{satdeg}
\newop\supp{supp}
\newop\Neg{Neg}
\newop\Nef{Nef}
\newop\Nefh{Nef_H}
\newop\Eff{Eff}
\newop\Zar{Zar}
\newop\MB{MB}
\newop\MBxC{MB\mathit{(x,C)}}
\newop\NnB{NnB}
\newop\Bigg{Big}
\newop\Effbar{\overline{\Eff}}
\newcommand\eqnref[1]{(\ref{#1})}

\def\keywordname{{\bfseries Keywords}}%
\def\keywords#1{\par\addvspace\medskipamount{\rightskip=0pt plus1cm
\def\and{\ifhmode\unskip\nobreak\fi\ $\cdot$
}\noindent\keywordname\enspace\ignorespaces#1\par}}
\def\subclassname{{\bfseries Mathematics Subject Classification
(2000)}\enspace}
\def\subclass#1{\par\addvspace\medskipamount{\rightskip=0pt plus1cm
\def\and{\ifhmode\unskip\nobreak\fi\ $\cdot$
}\noindent\subclassname\ignorespaces#1\par}}

\begin{document}
\title{Minkowski bases on algebraic surfaces with rational polyhedral pseudo-effective cone}
\author{Piotr Pokora\thanks{Research of Pokora was partially done during his PhD Internship in Warsaw Center of Mathematic and Computer Sciences.},
Tomasz Szemberg\thanks{Research of Szemberg was partially supported
by NCN grant UMO-2011/01/B/ST1/04875}}

\maketitle
\thispagestyle{empty}

\begin{abstract}
   The purpose of this note is to study the number of elements in Minkowski bases
   on algebraic surfaces with rational polyhedral pseudo-effective cone.
\subclass{14C20, 14M25}
\end{abstract}

\section{Introduction}
\label{sec:intro}
   Lazarsfeld and Musta\c t\u a  in \cite{LM09} and Kaveh and Khovanskii in \cite{KK} initiated a systematic study
   of Okounkov bodies. These are convex bodies $\Delta(D)$ in $\R^n$
   attached to big divisors $D$ on a smooth projective variety $X$
   of dimension $n$. They depend on the choice of a flag of subvarieties
   $(Y_n,Y_{n-1},\ldots,Y_1)$ of codimensions $n,n-1,\ldots,1$ in X respectively,
   such that $Y_n$ is a non-singular point of each of the $Y_i$'s. We refer to
   \cite{LM09} for details of the construction and a very enjoyable
   introduction to this circle of ideas.

   Okounkov bodies are subject of intensive ongoing research. \L uszcz-\'Swidecka
   observed in \cite{Pat} that for a del Pezzo surface there are finitely many
   basic bodies, called the Minkowski basis, such that all other bodies are obtained as their Minkowski sums
   (therefor the name of the basis).
   Building upon these ideas, \L uszcz-\'Swidecka and Schmitz introduced in \cite{PatDav}
   an effective algorithmic construction of Minkowski bases for algebraic surfaces
   with rational polyhedral pseudo-effective cone $\Effbar(X)$.

   In the present note we consider a natural question of how many elements there are
   in a Minkowski bases in the set-up of \cite{PatDav}. The answer is closely related to the partition, governed by Zariski decompositions, of the
   big cone of arbitrary smooth projective surfaces introduced in \cite{BKS04}.
\section{Preliminaries}
\label{sec:prelim}
   In this section we introduce the notation and collect some basic ideas underlying the present note.
   By a \emph{curve} we mean here an irreducible and reduced complete subscheme of dimension $1$.
   For a divisor $D$ on a smooth projective surface $X$ we denote by $D^\perp$ the set
   of curves intersecting $D$ with multiplicity zero, i.e. $D^\perp:=\left\{C\subset X:\; D.C=0\right\}$.

   We begin with a tool fundamental for understanding linear series
   on algebraic surfaces.
\begin{definition}[Zariski decomposition]
   Let $X$ be a smooth projective surface and let $D$ be a pseudo-effective $\Q$--divisor on $X$.
   Then there exist $\Q$--divisors $P_D$ and $N_D$ such that
   \begin{itemize}
   \item[a)] $D=P_D+N_D$;
   \item[b)] $P_D$ is a nef divisor and $N_D$ is either zero or it is supported
      on a union of curves $N_1,\ldots,N_r$ with negative definite
      intersection matrix;
   \item[c)] $N_i\in(P_D)^{\perp}$ for each $i=1,\ldots,r$.
   \end{itemize}
\end{definition}
   Let $(x,C)$ be a flag on a surface $X$. Let $D$ be a big divisor on $X$ with
   Zariski decomposition $D=P_D+N_D$. Lazarsfeld and Musta\c t\u a give in \cite{LM09}
   the description of $\Delta(D)$ as the area enclosed between the graphs
   of functions $\alpha(t)$ and $\beta(t)$ defined for real numbers $t$
   between $0$ and $\sup\left\{s\in\R:\; D-sC\; \mbox{ is effective}\right\}$
   as follows.
   $$\alpha(t)=\ord_x(N_{D-tC})\;\;\mbox{ and }\;\; \beta(t)=\alpha(t)+ {\rm vol}_{X|C}(P_{D-tC}) = \alpha(t) + P_{D-tC}\cdot C.$$
   Recently, the authors of \cite{PatDav} presented a different approach
   to describing Okounkov bodies for a certain class of smooth complex projective surfaces.
\begin{theorem}[\L uszcz-\'Swidecka, Schmitz]
\label{Mink}
   Let $X$ be a smooth complex projective surface with $\overline{{\rm Eff}(X)}$ rational polyhedral. Given a flag $(x,C)$,
   where $x$ is a general point and $C$ is a big and nef curve on $X$
   there exists a finite set of nef divisors $\MBxC = \{P_{1}, ..., P_{r}\}$ such that for a big and nef $\mathbb{R}$-divisor $D$ there exist
   uniquely determined non-negative real numbers $a_{i} \geq 0$ with
$$D = \sum_{i=1}^{r} a_{i}P_{i} \,\,\,\,\, \text{ and } \,\,\,\,\, \triangle(D) = \sum_{i=1}^{r} a_{i} \triangle(P_{i}),$$
   where the first sum indicates the numerical equivalence of divisors
   and the second sum is the Minkowski sum of convex bodies.
\end{theorem}
   The Theorem above justifies the following definition.
\begin{definition}[Minkowski basis]
   The set $\MBxC$ in Theorem \ref{Mink} is called the \emph{Minkowski basis} of $X$
   with respect to the flag $(x,C)$.
\end{definition}
\begin{remark}\label{rmk:minkowski basis not basis}
   Note that in general $\MBxC$ is not a basis of the Neron-Severi space $N^1(X)_{\R}$
   (treated as an $\R$--vector space).
\end{remark}
   The proof of Theorem \ref{Mink} in \cite{PatDav} gives in particular a simple way to construct Minkowski basis elements
   based on the  Bauer - K\"uronya - Szemberg decomposition of the big cone ${\rm Big}(X)$ \cite{BKS04}.
\begin{theorem}[BKS--decomposition]
   Let $X$ be a smooth complex projective surface. Then there is a locally finite decomposition of the big cone of $X$ into rational locally polyhedral subcones $\Sigma$ such that
   in the interior of each subcone $\Sigma$ the support $\Neg(\Sigma)$ of the negative part of the Zariski decomposition of the divisors in the subcone is constant.
\end{theorem}
   Now, the idea of \L uszcz-\'Swidecka and Schmitz is to assign to a chamber $\Sigma$
   an element in the Minkowski basis $M_{\Sigma}$. Specifically, let $C$ be a big and nef curve in the interior of $\Sigma$.
   Then $M_{\Sigma}= dC + \sum_{i=1}^r a_{i}N_i$, where $N_1,\dots,N_r \in \Neg(\Sigma)$ and $a_{i}$ are coefficients, which are the solution of the following system of equations
   \begin{equation}
   \label{minkowskicoefficients}
   S(a_{1}, ..., a_{r})^{T} = -d(C.N_{1}, ..., N_{r})^{T},
   \end{equation}
   where $S$ is the $r \times r$ intersection matrix of negative curves $N_{1}, ..., N_{r}$.
   Since $S$ is negatively defined, thus by an auxiliary result in \cite{BKS04} the inverse matrix $S^{-1}$ has only negative entries and thus all numbers $a_{i}$ are non-negative.

   It is convenient to work in the sequel with a compact slice $\Nefh(X)$ of the nef cone $\Nef(X)$
   defined as
   $$\Nefh(X)=\left\{D\in\Nef(X):\; D.H=1\right\}$$
   for a fixed ample divisor $H$ on $X$. We denote by $f_i$
   the number of $i$--dimensional faces of $\Nefh(X)$ for $i=0,\ldots,\rho(X)-1$.
   Moreover we write $f_0=(f_0)_b+(f_0)_{nb}$,
   where $(f_0)_b$ is the number of big vertices in $\Nefh(X)$
   and $(f_0)_{nb}$ is the number of non-big vertices, see also Lemma \ref{lem:nef non big}.

   Finally, we write $\NnB(X)$ for the number of numerical equivalence classes
   of nef and non-big integral divisors in $\Nefh(X)$ and we write $\Zar(X)$ for the
   number of Zariski chambers in the BKS-decomposition of $\Bigg(X)$.
\section{The cardinality of Minkowski bases}
   In the view of Remark \ref{rmk:minkowski basis not basis} it is natural to ask how many elements there are
   in the Minkowski basis. We will show here that the answer depends on the choice of the flag and that the number
   \begin{equation}\label{eq:number of elements}
       1 + \NnB(X) + \Zar(X)
   \end{equation}
   is a sharp upper bound for the number of elements in the Minkowski basis.
   The number of negative curves on surfaces with $\overline{{\rm Eff}(X)}$ rational polyhedral
   is finite, hence the number of Zariski chambers on such surfaces is finite as well.
   This number can be large. For example, for del Pezzo surfaces $X_{i}$ obtained as the blow ups of $\mathbb{P}^2$
   in $i \in \{1, ..., 8\}$ general points we have
 \begin{equation}\label{eq:number of chambers on del pezzo}
      \begin{array}{c|c|c|c|c|c|c|c|c}
         i        & 1 & 2 & 3 & 4 & 5 & 6 & 7 & 8  \\ \hline
         \Zar(X_{i})   & 2 & 5 & 18 & 76 & 393 & 2\, 764 & 33\, 645 & 1\, 501\, 681
      \end{array},
\end{equation}
   see \cite{BFN10}.

   Now, we explain that the second summand in \eqnref{eq:number of elements} is also finite.
\begin{lemma}[Nef, non-big divisors]\label{lem:nef non big}
   Let $X$ be a surface with $\overline{{\rm Eff}(X)}$ rational polyhedral.
   Then there is only a finite number of nef and non-big divisors in $\Nefh(X)$.
\end{lemma}
\proof
  Assume to the contrary that there are two divisors $N_{1}, N_{2}$, which are nef and not big,
  such that for all $t \in [0, 1]$ the divisors $tN_{1} + (1-t)N_{2}$ lie on the common face
  (here the rational polyhedrality assumption comes into the play). Thus
   $(tN_{1} + (1-t)N_{2})^2 = 0 $ for every $t \in [0, 1]$, what implies that $N_{1}.N_{2} = 0$.
   It means that the intersection matrix of $N_{1}, N_{2}$ is the zero matrix of size $2 \times 2$,
   which contradicts the index theorem.
\endproof
   Now we relate the number in \eqnref{eq:number of elements} to the geometry of the solid $\Nefh(X)$.
\begin{proposition}
   Let $X$ be a smooth complex projective variety with $\overline{{\rm Eff}(X)}$ rational polyhedral.
   Then
   \begin{equation}\label{eq:formula A}
      \sum_{i=0}^{\rho-1} f_{i} = 1 + \NnB(X) + \Zar(X).
   \end{equation}
\end{proposition}
\begin{proof}
   Let $G$ be a face of $\Nefh(X)$. If $G=\Nefh(X)$, then this corresponds to $f_{\rho-1}=1$
   and is accounted for by $1$ on the right side in the formula \eqnref{eq:formula A}.
   Otherwise we distinguish two cases:
   either $G$ is a vertex of $\Nefh(X)$ which is not big, hence $G^2=0$
   or $G$ is a big vertex or a face of dimension $\geq 1$.

   The first case occurs $(f_0)_{nb}$ times and is accounted for
   by the second summand on the right in \eqnref{eq:formula A}.

   The second case corresponds to the third summand in \eqnref{eq:formula A}.
   Indeed, given a nef and big divisor $D$ there exists a Zariski chamber
   $\Sigma_D$ with $\Neg(\Sigma_D)=D^\perp$. This follows from
   Nakamaye's result \cite[Theorem 1.1]{Nak00}.
   Thus the inequality $\leq $ in \eqnref{eq:formula A} is established.

   For the reverse inequality it suffices to show that
   distinct Zariski chambers determine distinct faces of $\Nefh(X)$.
   To this end let $\Sigma$ be a Zariski chamber.
   By \cite{BKS04} there is a face of $\Nefh(X)$ orthogonal to the
   support of $\Neg(\Sigma)$. The injectivity of this assigment
   $\Sigma \to \Neg(\Sigma)^\perp$ follows again from the aforementioned
   result of Nakamaye.
\end{proof}
   Now we are in a position to prove our main result.
\begin{theorem}\label{thm:main}
   Let $X$ be a smooth complex projective variety with $\overline{{\rm Eff}(X)}$ rational polyhedral.
   Given a flag $(x, A)$, where $A$ is an ample curve and $x$ is a smooth point on $A$, there is
   $$\#\MB(x,A) = 1 + \NnB(X) + \Zar(X).$$
\end{theorem}
\begin{proof}
   Given Zariski chambers $\Sigma_1$, $\Sigma_2$ with $\Neg(\Sigma_1)=\left\{N_1,\ldots,N_{n} \right\}$ and $\Neg(\Sigma_2)=\left\{N_{n+1},\ldots,N_{m} \right\}$, one associates to them the Minkowski basis elements
   $$M_{\Sigma_1}= b_{1}A + \sum_{j=1}^{n} a_{i}N_j \,\,\, \text{ and } \,\,\, M_{\Sigma_2}=b_{1}A + \sum_{j=n+1}^{m} a_{i}N_j.$$
   Suppose that $\Neg(\Sigma_{1}) \neq \Neg(\Sigma_{2})$ and assume to the contrary that $M_{\Sigma_{1}} = M_{\Sigma_{2}}$.  Furthermore we may assume after reordering that the symmetric difference between these negative supports is $\{N_{k}, ..., N_{m}\}$ for a certain $k \in \{1, ..., m\}$. By the construction of Minkowski basis elements \cite{PatDav} we know that $$M_{\Sigma_{1}} = M_{\Sigma_{2}} \in N_{i}^{\perp} \,\,\, \text{ for all } \,\,\, i \in \{1, ..., m \}.$$
   This implies that for every $N_{i}$ we have  $N_{i}.M_{\Sigma_{1}} = N_{i}.M_{\Sigma_{2}} = 0$. Let us take one of the elements from $\{N_{k}, ..., N_{m} \}$. These implies in particular that $N_{r}.A = 0$, a contradiction. Proceed in the same spirit one shows that the symmetric difference is empty and $\Neg(\Sigma_{1}) =\Neg(\Sigma_{2})$, what ends the proof.
\end{proof}

\begin{example}[Del Pezzo surfaces]
   Using the above theorem we can compute the cardinality of Minkowski basis for del Pezzo surfaces $X_{i}$ with respect to a fixed ample flag $(x,A)$.
   To this end we need to compute the number of nef non-big curves on $X_i$. Let
   $C=aH-\sum b_jE_j$ be such a curve, where as usually $\pi_i:X_i\to\P^2$
   is the blow up of $\mathbb{P}^2$ at $i$ general points with exceptional divisors $E_1,\ldots,E_i$ and
   $H=\pi_i^*(\calo_{\P^2}(1))$.
   First we observe that $C$ is a rational curve. This follows from the adjunction since
   $$2(p_a(C)-1)=K_{X_i}\cdot C+C^2=K_{X_i}\cdot C<0$$
   implies $p_a(C)=0$. Hence
\begin{equation}\label{eq:pa equal zero}
   2=-K_{X_i}\cdot C=3a-\sum b_j.
\end{equation}
   On the other hand
\begin{equation}\label{eq:C2 equal zero}
   0=C^2=a^2-\sum b_j^2.
\end{equation}
   It is elementary to check that \eqnref{eq:pa equal zero} and \eqnref{eq:C2 equal zero}
   have only finitely many integral solutions, listed (up to permutation) in the following table
$$\begin{tabu}{c|ccccccccc}
   & a & b_1 & b_2 & b_3 & b_4 & b_5 & b_6 & b_7 & b_8 \\
   \hline
   C(1) & 1 & 1 & 0 & 0 & 0 & 0 & 0 & 0 & 0 \\
   C(2) & 2 & 1 & 1 & 1 & 1 & 0 & 0 & 0 & 0 \\
   C(3) & 3 & 2 & 1 & 1 & 1 & 1 & 1 & 0 & 0 \\
   C(4) & 4 & 2 & 2 & 2 & 1 & 1 & 1 & 1 & 0 \\
   C(5) & 5 & 3 & 2 & 2 & 2 & 1 & 1 & 1 & 1 \\
   C(6) & 6 & 3 & 3 & 2 & 2 & 2 & 2 & 1 & 1 \\
   C(7) & 7 & 3 & 3 & 3 & 3 & 2 & 2 & 2 & 1 \\
   C(8_a) & 8 & 4 & 3 & 3 & 3 & 3 & 2 & 2 & 2 \\
   C(8_b) & 8 & 3 & 3 & 3 & 3 & 3 & 3 & 3 & 1 \\
   C(9) & 9 & 4 & 4 & 3 & 3 & 3 & 3 & 3 & 2 \\
   C(10) & 10 & 4 & 4 & 4 & 4 & 3 & 3 & 3 & 3 \\
   C(11) & 11 & 4 & 4 & 4 & 4 & 4 & 4 & 4 & 3
\end{tabu}$$
   Note that all solutions can be obtained from $C(1)$ applying standard Cremona
   transformations. This verifies again that an irreducible nef non-big curve
   on a del Pezzo surface is rational.

   Counting all curves $C(j)$ on the appropriate surface $X_i$ and taking
   \eqnref{eq:number of chambers on del pezzo} into account we have
  $$
      \begin{array}{c|c|c|c|c|c|c|c|c}
         i        & 1 & 2 & 3 & 4 & 5 & 6 & 7 & 8  \\ \hline
         \#\MB(x,A)   & 3 & 7 & 21 & 81 & 403 & 2\, 797 & 33\, 764 & 1\, 503\, 721
      \end{array}.
 $$
\end{example}
\begin{remark}\rm
   For $r=0,\dots,8$ let $X_r$ be a del Pezzo surface arising as the blow up of the projective plane $\P^2$
   in $r$ general points. Let $C$ be a curve in the anti-canonical system $-K_{X_r}$.
   There is a Weyl group action on $\Effbar(X)$, which fixes the anti-canonical class,
   see \cite{Man}. In this situation, there is a Weyl invariant Minkowski basis $\MBxC$.
   Indeed, it can be constructed taking for each $j=0,\ldots,\rho(X)-1$
   an element $M_j$ corresponding to a $j$--dimensional face of $\Nefh(X)$
   and then all of its images under the action of the Weyl group, see also \cite[Section 3.1]{BKS04}.
\end{remark}
   Now we show that for a special choice of a flag $(x,C)$, it might happen that
   the number of divisors in the Minkowski basis is strictly smaller than
   the number in \eqnref{eq:number of elements}. In fact we get Minkowski bases with any number
   of elements between $3$ and $7$ on the del Pezzo surface $X_2$.
\begin{example}
   For del Pezzo surface $X_{2}$ we have the following possibilities:
\begin{itemize}
\item Fix a toric flag for $X_{2}$, i.e. $(x,L_1)$ with $L_1\in|H-E_1|$ and $x=L_1\cap L_2$ for a fixed line $L_2\in|H - E_{2}|$.
   Then by \cite{PSU13}
   $$\MB(x,L_1) = \{H, H - E_{1}, H - E_{2} \}.$$
\end{itemize}
   From now on $x$ denotes a general point on the flag curve $C$.
\begin{itemize}
\item For the flag $(x, C)$, where $C \in|H|$, we have
   $$\MBxC = \{H, H - E_{1}, H - E_{2}, 2H - E_{1} - E_{2} \}.$$
\item For a curve $C \in |2H-E_{1}|$, we get
   $$\MBxC = \{ 2H-E_{1}, H - E_{1}, H - E_{2}, H, 3H - 2E_{1} - E_{2} \}.$$
\item For a curve $C \in |2H - E_{1} - E_{2}|$ we have
   $$\MBxC = \{2H - E_{1} - E_{2}, H - E_{1}, H - E_{2}, 2H - E_{1}, 2H - E_{2} , H \}.$$
\item For the anticanonical flag $(x,C)$ with a curve $C\in| - K_{X_2}|$ we have
   $$\MBxC = \{ -K_{X_2}, H, H-E_{1}, H-E_{2}, 2H - E_{1} - E_{2}, 3H-E_{1}, 3H-E_{2} \}.$$
\end{itemize}
\end{example}
\begin{remark}\rm
   It would be interesting to know effective lower bounds on the number
   of elements in the Minkowski basis.
\end{remark}
\paragraph*{\emph{Acknowledgement.}}
   The first author would like to express his gratitude to Jarek Wi\'sniewski for many interesting conversations
   and for suggesting several problems related to Okounkov bodies.
   Both authors would like to thank Polish Railways. This research has originated
   during a long train journey from Warsaw to Cracow.

   Also authors would like to express their gratitude to David Schmitz for pointing out a mistake in the previous version of this note.


\bigskip \small

\bigskip
   Piotr Pokora,
   Instytut Matematyki,
   Pedagogical University of Cracow,
   Podchor\c a\.zych 2,
   PL-30-084 Krak\'ow, Poland.

\nopagebreak
   \textit{E-mail address:} \texttt{piotrpkr@gmail.com}

\bigskip
   Tomasz Szemberg,
   Instytut Matematyki,
   Pedagogical University of Cracow,
   Podchor\c a\.zych 2,
   PL-30-084 Krak\'ow, Poland.

\nopagebreak
   \textit{E-mail address:} \texttt{tomasz.szemberg@gmail.com}


\end{document}